\documentclass[reqno]{amsart}
\usepackage{amsmath}
\usepackage[dvips]{graphicx}
\usepackage{amsfonts}
\usepackage{amssymb}
\usepackage{latexsym}
\graphicspath{{Img/}}
\newtheorem{theorem}{Theorem}
\theoremstyle{plain}

\newtheorem{corollary}{Corollary}

\newtheorem{lemma}{Lemma}

\newtheorem{problem}{Problem}

\numberwithin{equation}{section}

\begin{document}

\title[D'Alembert-Lagrange principle yields a weighted balancing condition]{D'Alembert-Lagrange principle for point masses yields a system of weighted balancing unit vectors  in the three dimensional Euclidean Space}
\author{Anastasios N. Zachos}
\address{University of Patras, Department of Mathematics, GR-26500 Rion, Greece}
\email{azachos@gmail.com}
\keywords{D'Alembert-Lagrange principle, point masses, weighted Fermat problem, weighted Fermat-Torricelli point, weighted balancing condition} \subjclass[2010]{Primary 70S05;37J06 Secondary 51M16}

\dedicatory{To Lagrange's superposition of equilibria}

\begin{abstract}
In this paper, we prove that D'Alembert-Lagrange principle for point masses using Lagrange-Mach's mechanical construction yields a weighted balancing condition of unit vectors in $\mathbb{E}^{3}.$
\end{abstract}\maketitle

\section{Introduction}

We start by giving D'Alembert-Lagrange principle for a system of points $x_{i}$ with masses $m_{i},$ which move freely in $\mathbb{E}^{3}$  in the configuration manifold $\mathbb{R}^{3n}$ for $i=1,2,\ldots,n$ (\cite[Section~21,pp.~91-97]{Arnold:89},\cite{Lagrange:}).

We note that tangent vectors to the configuration manifold are called virtual variations and the constraint forces are defined by $\textbf{R}_{i}:=m_{i}\ddot{x}_{i}+\frac{\partial U}{\partial x_{i}}.$

We denote by $TM_{x}$ the tangent space of the surface $M$ in the three dimensional space $\{x\}.$

The D'Alembert-Lagrange principle states that:

The sum of the works of the constrained forces on any virtual variation $\xi_{i}\in TM_{x}$ is zero or
$\sum_{i=1}^{n}(m_{i}\ddot{x}_{i}+\frac{\partial U}{\partial x_{i}},\xi)=0.$

We recall the weighted Fermat Problem for $m$ non-collinear and non-coplanar points $A_{i}(x_{i},y_{i},z_{i})$  in $\mathbb{E}^{3}$ (\cite{Sturm:84},\cite{Lindelof:67},\cite{Kup/Mar:97} and \cite{BolMa/So:99}), which was first states in the 17th century by Fermat for three unweighted points (\cite{Fermat:91}):

\begin{problem}[The weighted Fermat problem in $\mathbb{R}^{3}$]\label{WFN}
Find $A_{0}(x,y,z)$ in $\mathbb{R}^{3},$ such that:
\begin{equation}\label{objectivewfrn}
f(\{A_{0}\})=\sum_{i=1}^{m}b_{i}\sqrt{(x-x_{i})^2+(y-y_{i})^2+(z-z_{i})^2}\to min.
\end{equation}
\end{problem}

We need the following fundamental lemma, which deals with the characterization of solutions for the weighted Fermat problem (\cite{Kup/Mar:97},\cite{BolMa/So:99}). The unweighted case was first proved by Lindelof and Sturm in \cite{Lindelof:67} and \cite{Sturm:84}.



\begin{lemma}{\cite[Theorem~18.37,pp.~250]{BolMa/So:99},\cite{Kup/Mar:97}}\label{characteizationFTsolution}
(I) The weighted Fermat point $A_{0}$
exists and is unique.

(II) If for each point  $A_{i}\in \{A_{1},A_{2},...,A_{m}\}$
\[ \|\sum_{j=1,j\ne i}^{m}b_{j}u(A_{j},A_{i})\|>b_i, \]
for $i,j=1,2,\ldots,m,$ then

(a) the weighted Fermat (Fermat-Torricelli) point $A_{0}$ does not belong
to\\ $\{A_{1},A_{2},...,A_{m}\}$

(b) The following weighted balancing condition of unit vectors $u(A_{0},A_{i})$ occurs:

\[ \sum_{i=1}^{m}b_{i}u(A_{0},A_{i})=0\]
(weighted Fermat-Torricelli tree solution).
\end{lemma}


In this paper, we prove that D'Alembert-Lagrange principle for point masses having only one point that moves freely in $E^{3}$ implies a weighted balancing condition of unit vectors in the spirit of weighted Fermat-Torricelli tree solution in $\mathbb{E}^{3}.$

\section{The equivalence of D'Alembert-Lagrange principle and Fermat-Torricelli's weighted balancing condition in two dimensions}
In this section, we generalize Lagrange-Mach's mechanical construction for $m$ unequal weights in the Eucidean plane $\mathbb{E}^{2}$ and on a smooth surface $M$ in $\mathbb{E}^{3}$ and we prove the equivalence of D'Alembert-Lagrange principle and Fermat-Torricelli's weighted balancing condition.

Let $\{A_{1},A_{2},\ldots, A_{m}\}$ be $m$ fixed points and $A_{0}$ be a point in the Euclidean plane $\mathbb{E}^{2}$ or an horizontal plane and let pulleys be placed at these fixed points over which $m$ strings are passed loaded with weights $b_{1},b_{2},\ldots,b_{m},$ respectively and knotted at $A_{0}.$
We set $\alpha_{i0j}\equiv \angle A_{i}A_{0}A_{j}$  and the infinitesimal small distances $\delta \xi_{j}\equiv A_{0}A_{0^{\prime},j},$ such that
$A_{0^{\prime},j}\in [A_{0},A_{j}].$
We denote by $A_{0^{\prime},ij}$ the trace of the orthogonal projection of $A_{0^{\prime},j}$ to the line defined by $A_{0},A_{i}$ and by $A_{i^{\prime}}$ the symmetrical of $A_{i}$ with respect to $A_{0}.$

Furthermore, we set:

$c_{i,j}\equiv \begin{cases} -\cos\alpha_{j0i^{\prime}},& \text{if $A_{0^{\prime},ij}$
lies on the line segment [$A_{0},A_{i^{\prime}}$] },\\
\cos\alpha_{j0i},& \text{if $A_{0^{\prime},ij}$ lies on the line segment [$A_{0},A_{i}$]},
\end{cases}$

where $\alpha_{j0i^{\prime}}=\pi-\alpha_{j0i},$ for $i,i^{\prime},j=1,2,\ldots,m$ and $0<\alpha_{j0i^{\prime}},\alpha_{j0i}<\pi.$

\begin{theorem}\label{DalembertLagrangeMachR2}
D'Alembert Lagrange principle for Lagrange-Mach's mechanical construction of pulleys in $\mathbb{E}^{2}$ implies that:
\begin{equation}\label{sumweightedunitvectors}
\sum_{i=1}^{m}b_{i}u(A_{0},A_{i})=0.
\end{equation}
\end{theorem}

\begin{proof}

By displacing the point $A_{0}$ in the directions of $A_{0}A_{j}$ by the infinitesimal small distances $\delta \xi_{j},$ we obtain:

\[\sum_{j=1}^{m}\sum_{i=1}^{m}c_{i,j}b_{i}\delta \xi_{i}=0,\]

which gives

\begin{equation}\label{sumcosr2}
\sum_{i=1}^{m}b_{i} \cos\alpha_{j0i}=0,
\end{equation}

for every $j=1,2,\ldots,m.$

By taking a fixed index $k\in \{1,2,\ldots,m\}$ and the addition of angles $\angle A_{j0k}+\angle A_{k}A_{0}A_{i}$ in $\mathbb{E}^{2},$ we get:  
\[\sum_{i=1}^{m}b_{i}\cos(\alpha_{j0k}+\alpha_{k0i})=0\]

or

\[\sum_{i=1}^{m}b_{i}\cos\alpha_{j0k}\cos\alpha_{k0i}-\sum_{i=1}^{m}b_{i}\sin\alpha_{j0k}\sin\alpha_{k0i}=0,\]

which gives

\[\sum_{i=1}^{m}b_{i}\sin\alpha_{k0i}=0.\]

Thus, for $k=j,$ we get:

\begin{equation}\label{sumsinr2}
\sum_{i=1}^{m}b_{i}\sin\alpha_{j0i}=0
\end{equation}

By taking into account (\ref{sumcosr2}),(\ref{sumsinr2}) we derive (\ref{sumweightedunitvectors}).

\end{proof}

\begin{corollary}{Lagrange-Mach mechanical contsruction \cite[pp.~61-62]{Mach:19}}
For $b_{1}=b_{2}=b_{3}>0$ and $b_{4}=\ldots=b_{m}=0,$ we get
\[\sum_{i=1}^{3}u(A_{0},A_{i})=0.\]
\end{corollary}

Let $\{A_{1},A_{2},\ldots, A_{m}\}$ be $m$ fixed points and $A_{0}$ be a point on a smooth surface $M$ in $\mathbb{E}^{3}$  and let pulleys be placed at these fixed points over which $m$ strings are passed loaded with weights $b_{1},b_{2},\ldots,b_{m},$ respectively and knotted at $A_{0}.$ Furhtermore, the loose ends are passed through $m$ holes, which correspond to the fixed points $\{A_{1},A_{2},\ldots, A_{m}\}$ and are attached to the physical weights
$b_{1},b_{2},\ldots,b_{m},$ respectively.

We assume the existence of a small positive real number $I,$ such that the  injectivity radius of $M$ $inj(M):=I.$ Therefore, we may consider a bounded and connected region $D$ on $M$, which is subset of a disk $B_{X;I}$ with center $X$ and radius $R=I,$ such that $\{A_{0,}A_{1},A_{2},\ldots, A_{m}\}$ in $D.$
Hence, any geodesic arc that connects $\{A_{0,}A_{1},A_{2},\ldots, A_{m}\}$ with any point $Y$ in $D$ is a shortest arc (segment) (\cite[p.~42]{Berger:03}).
Thus, we define the exponential map at $A_{0}$ $exp_{A_{0}}:W_{A_{0}}\subset T_{A_{0}}(M)\to M$ from the vector $w$ starting at $A_{0}$ travelling along up to the point $ \gamma_{\frac{w}{||w||}}(||w||)$ having length $||w||$ (\cite[pp.~222-223]{Berger:03}).

\begin{theorem}\label{DalembertLagrangeMachM}
D'Alembert Lagrange principle for Lagrange-Mach's mechanical construction of pulleys on a smooth surface $M$ in $\mathbb{E}^{3}$ implies that:
\begin{equation}\label{sumweightedunitvectorssurface}
\sum_{i=1}^{m}b_{i}\frac{exp_{A_{0}}^{-1}(A_{0},A_{i})}{||exp_{A_{0}}^{-1}(A_{0},A_{i})||}=0.
\end{equation}
\end{theorem}

\begin{proof}
By applying Lagrange-Mach's mechanical construction of pulleys on a smooth surface $M$ in $\mathbb{E}^{3},$ we obtain equilibrium at $A_{0},$ such that
the $m$ strings are attached at the geodesic arcs $A_{i}A_{0}.$ By taking the inverse of the exponential mapping at $A_{0}$ $exp_{A_{0}}^{-1}(A_{i}),$ we may move to $T_{A_{0}}(M).$ Hence, by applying Theorem~\ref{DalembertLagrangeMachR2}, we obtain a system of weighted balancing unit vectors at $T_{A_{0}}(M),$ which gives (\ref{sumweightedunitvectorssurface}).   
\end{proof}

\section{The equivalence of D'Alembert-Lagrange and Fermat-Torricelli's weighted balancing condition in three dimensions}
In this section, we generalize Lagrange-Mach's mechanical construction for $m$ unequal weights in the three dimensional Eucidean Space $\mathbb{E}^{3}$ and we prove the equivalence of D'Alembert-Lagrange principle and Fermat-Torricelli's weighted balancing condition.

Let $\{A_{1},A_{2},\ldots, A_{m}\}$ be $m$ fixed points and $A_{0}$ be a point in $\mathbb{E}^{3}$ and let pulleys be placed at these fixed points over which $m$ strings are passed loaded with weights $b_{1},b_{2},\ldots,b_{m},$ respectively and knotted at $A_{0}.$
We set the infinitesimal small distances $\delta \xi_{j}\equiv A_{0}A_{0^{\prime},j},$ such that
$A_{0^{\prime},j}\in [A_{0},A_{j}].$

We denote by $A_{i^{\prime}}$ the symmetrical point of $A_{i}$ with respect to $A_{0},$ by $A_{k,i0j}$ the trace of the orthogonal projection of $A_{k}$ to the plane defined by $\triangle A_{i}A_{0}A_{j},$ by $A_{k^{\prime},i0j}$ the symmetrical point of $A_{k,i0j}$ with respect to $A_{0},$ by $A_{0^{\prime\prime},(0k^{\prime},i0j)}$ the trace of the orthogonal projection of $A_{0^{\prime},j}$ to the line defined by the segment $A_{0}A_{k^{\prime},i0j}$ and by $A_{0^{\prime\prime\prime},0k^{\prime}}$ the trace of the orthogonal projection of $A_{0^{\prime\prime},(0k^{\prime},i0j)}$
to the line defined by the segment $A_{0}A_{k^{\prime}},$ for $i,j,k,i^{\prime},k^{\prime}=1,2,\ldots,m.$

We set:

$\alpha_{i,j0k}\equiv \angle A_{i}A_{0}A_{i,j0k},$

$\omega_{i,j0k}\equiv \angle A_{i,j0k}A_{0}A_{j},$

$c_{i,j0k}=\begin{cases} -\delta \xi_{k}\cos\omega_{i,j0k}\cos\alpha_{i,j0k},& \text{if $A_{0^{\prime\prime\prime},0i} \notin$ [$A_{0},A_{i}$] },\\
\delta \xi_{k}\cos\omega_{i,j0k}\cos\alpha_{i,j0k},& \text{if $A_{0^{\prime\prime\prime},0i}\in [A_{0},A_{i}]$}.
\end{cases}$


\begin{theorem}\label{DalembertLagrangeMachR3}
D'Alembert Lagrange principle for Lagrange-Mach's mechanical construction of pulleys in $\mathbb{E}^{3}$ implies that:
\begin{equation}\label{sumweightedunitvectorsr3}
\sum_{i=1}^{m}b_{i}u(A_{0},A_{i})=0.
\end{equation}
\end{theorem}

\begin{proof}

We consider the following system of unit vectors $u(A_{0},A_{i})$ in $\mathbb{E}^{3}:$

We express the unit vectors $\vec {u}(A_{0},A_{i})$ for $i=1,2,3,4,$ using spherical coordinates:

\begin{equation}\label{spherical1}
u(A_{0},A_{j})=(1,0,0),
\end{equation}

\begin{equation}\label{spherical2}
u(A_{0},A_{k})=(\cos\alpha_{j0k},\sin\alpha_{j0k},0),
\end{equation}

\begin{equation}\label{spherical3}
u((A_{0},A_{i})=(\cos a_{i,j0k} \cos\omega_{i,j0k},\cos a_{i,j0k} \sin\omega_{i,j0k},\sin a_{i,j0k} ),
\end{equation}

$\ldots$

\begin{equation}\label{spherical4}
u(A_{0},A_{m})=(\cos a_{m,j0k} \cos\omega_{m,j0k},\cos a_{m,j0k} \sin\omega_{m,j0k},\sin a_{m,j0k} ).
\end{equation}

for $i,j,k=1,2,\ldots,m, i\neq j\ne k.$

By displacing the point $A_{0}$ in the directions of $A_{0}A_{j}$ by the infinitesimal small distances $\delta \xi_{j},$ we obtain:

\[\sum_{k=1}^{m}\sum_{i=1}^{m}c_{i,j0k}b_{i}\delta \xi_{k}=0,\]

which gives

\begin{equation}\label{sumcosr3}
\sum_{k=1}^{m}\sum_{i=1}^{m}b_{i} \cos\alpha_{i,j0k}\cos\omega_{i,j0k}=0,
\end{equation}

or

\begin{equation}\label{sumcosr3bis}
\sum_{i=1}^{m}b_{i} \cos\alpha_{i,j0k}\cos\omega_{i,j0k}=0,
\end{equation}

for every $j=1,2,\ldots,m.$

By substituting the inner product 
\[u(A_{0},A_{j})\cdot u(A_{0},A_{k})=\cos\alpha_{j0k}, \]
in (\ref{sumcosr3}), we obtain:

\begin{equation}\label{sumcosanglesr3}
\sum_{k=1}^{m}b_{i}\cos\alpha_{j0k}=0.
\end{equation}

Without loss of generality, we fix the indices $j=1, k=2$ and we consider the particular system of unit vectors:

\begin{equation}\label{spherical1102}
u(A_{0},A_{1})=(1,0,0),
\end{equation}

\begin{equation}\label{spherical2102}
u(A_{0},A_{2})=(\cos\alpha_{102},\sin\alpha_{102},0),
\end{equation}

$\ldots$

\begin{equation}\label{spherical3102}
u((A_{0},A_{k})=(\cos a_{k,102} \cos\omega_{k,102},\cos a_{k,102} \sin\omega_{k,102},\sin a_{k,102} ).
\end{equation}

for $k=1,2,\ldots, m.$

By replacing $j=2$ and $j=3$ in (\ref{sumcosanglesr3}),

we get

\begin{equation}\label{sumcosanglesr320k}
\sum_{k=1}^{m}b_{i}\cos\alpha_{20k}=0.
\end{equation}

and

\begin{equation}\label{sumcosanglesr320k2}
\sum_{k=1}^{m}b_{i}\cos\alpha_{30k}=0.
\end{equation}

By substituting the inner product \[u(A_{0},A_{2})\cdot u(A_{0},A_{k}) =\cos\alpha_{102}\cos a_{k,102}\cos\omega_{k,102}+\sin\alpha_{102}\cos a_{k,102} \sin\omega_{k,102}\]
in (\ref{sumcosanglesr320k}) and taking into account (\ref{sumcosr3bis}), we derive:

\begin{equation}\label{sumcossin102r3}
\sum_{k=1}^{m}b_{i} \cos\alpha_{k,102}\sin\omega_{k,102}=0.
\end{equation}

By substituting the inner product 
\[u(A_{0},A_{3})\cdot u(A_{0},A_{k}) =\cos a_{3,102}\cos\omega_{3,102}\cos a_{k,102}\cos\omega_{k,102}+\]\[+\cos a_{3,102} \sin\omega_{3,102}\cos a_{k,102} \sin\omega_{k,102}+\sin\alpha_{3,102}\sin\alpha_{k,102}\]

in (\ref{sumcosanglesr320k2}) and taking into account (\ref{sumcosr3bis}), (\ref{sumcossin102r3}), we derive:

\begin{equation}\label{sumsin102r3}
\sum_{k=1}^{m}b_{i} \sin\alpha_{k,102}=0.
\end{equation}
Hence,  we obtain the following system of three equations (\ref{sumcosr3bis}),(\ref{sumcossin102r3}), (\ref{sumsin102r3}):

\[\sum_{k=1}^{m}b_{i} \cos\alpha_{k,102}\cos\omega_{k,102}=0,\]

\[\sum_{k=1}^{m}b_{i} \cos\alpha_{k,102}\sin\omega_{k,102}=0,\]

\[\sum_{k=1}^{m}b_{i} \sin\alpha_{k,102}=0,\]

which yields (\ref{sumweightedunitvectorsr3}).

\end{proof}

\end{document}